\documentclass[12pt,letterpaper,titlepage]{amsart}
\usepackage{amsmath, amssymb, amsthm, amsfonts,amscd,xr}
\usepackage{graphicx}

\theoremstyle{plain}
\newtheorem{theorem}{Theorem}[section]
\newtheorem{proposition}[theorem]{Proposition}
\newtheorem{cor}[theorem]{Corollary}

\newtheorem{prop}[theorem]{Proposition}
\newtheorem{lemma}[theorem]{Lemma}
\newtheorem{definition}[theorem]{Definition}
\newtheorem{property}[theorem]{Property}
\theoremstyle{definition}

\newtheorem{rmk}[theorem]{Remark}
\numberwithin{equation}{section}

\newcommand{\C}{\mathbb{C}}

\newcommand{\R}{\mathbb{R}}

\newcommand{\SO}{\operatorname{SO}}

\newcommand{\SL}{\operatorname{SL}}

\newcommand{\Ad}{\operatorname{Ad}}

\newcommand{\diag}{\operatorname{diag}}

\usepackage[usenames]{color}

\def\af{\mathfrak{a}}

\def\gf{\mathfrak{g}}
\def\hf{\mathfrak{h}}
\def\kf{\mathfrak{k}}

\def\pf{\mathfrak{p}}
\def\qf{\mathfrak{q}}

\def\sf{\mathfrak{s}}

\def\so{\mathfrak{so}}

\def\1{{\bf1}}

\begin{document}
\title[decomposition theorems]
{Decomposition theorems for triple spaces}
\date{December 20, 2012}
\author[Danielsen, Kr\"{o}tz, Schlichtkrull]
{Thomas Danielsen$^1$, Bernhard Kr\"{o}tz$^2$ \and \\ Henrik Schlichtkrull$^1$}
\thanks{$^1$Department of Mathematical Sciences, University of Copenhagen, Denmark}
\thanks{$^2$Department of Mathematics, University of Paderborn, Germany}
\email{thd@math.ku.dk}
\email{bkroetz@math.uni-paderborn.de}
\email{schlicht@math.ku.dk}
\begin{abstract}
A triple space is a homogeneous space $G/H$ where 
$G=G_0\times G_0\times G_0$ is a threefold product
group and $H\simeq G_0$ the diagonal subgroup of $G$.
This paper concerns the geometry of the triple 
spaces with $G_0=\SL(2,\R)$,
$\SL(2,\C)$ or $\SO_e(n,1)$ for $n\ge 2$. 
We determine the abelian subgroups $A\subset G$
for which there is a polar decomposition $G=KAH$,
and we determine for which minimal parabolic subgroups 
$P\subset G$, the orbit $PH$ is open in $G/H$.
\end{abstract}
\maketitle

\section{Introduction}

Let $G_0$ be a real reductive group and let $G= G_0 \times G_0\times G_0$ and 
$H=\diag (G_0)$.  The corresponding homogeneous space 
$G/H$ is called a {\it triple space}. Triple spaces are 
examples of non-symmetric homogeneous spaces, as there is no involution
of $G$ with fixed point group $H$. It is interesting in the non-symmetric
setting to explore properties, which play an important role for 
the harmonic analysis of symmetric spaces. In this paper we 
examine the geometric structure of some triple spaces from this
point of view.

One important structural result for symmetric spaces is the polar
decomposition $G=KAH$. Here $K\subset G$ is a maximal compact
subgroup, and $A\subset G$ is abelian. Polar decomposition 
for a Riemannian symmetric space $G/K$ is due to Cartan, 
and it was generalized to reductive symmetric spaces 
in the form $G=KAH$ by Flensted-Jensen~\cite{FJ}. 

For triple spaces in general, the sum of the dimensions of
$K$, $A$ and $H$ can be strictly smaller than the dimension
of $G$, which obviously prevents $G=KAH$.
Here we are interested in the triple spaces with
\begin{equation}\label{G0}
G_0=\SL(2,\R), \,\SL(2,\C), \,\SO_e(n,1)\quad(n=2,3,\dots)
\end{equation}
for which there is no obstruction by dimensions.
In Theorem \ref{thm polar} we show that indeed these spaces admit
a polar decomposition as above, and we determine
precisely for which maximal split abelian subgroups $A$ 
the decomposition is valid. For the simplest choice of 
group $A$ we describe the indeterminateness of the $A$-component
for a given element in $G$, and we identify the invariant
measure on $G/H$ in these coordinates.

Another important structural result for a Riemannian symmetric 
space $G/K$ is the fact (closely related to Iwasawa decomposition)
that minimal parabolic subgroups $P$ act transitively. For 
non-Riemannian symmetric spaces there is no transitive action of $P$,
but it is an important result, due to Wolf \cite{Wolf}, that $P$ has an orbit 
on $G/H$ which is open. In Proposition \ref{thm sph} we verify
that this is the case also for the spaces in (\ref{G0}), and
we determine precisely for which minimal parabolic subgroups $P$
the orbit through the origin is open.

By combining these results we conclude in
Corollary \ref{strong spherical} that 
there exist maximal split abelian 
subgroups $A$ for which $G=KAH$ and for which
$PH$ is open for all minimal parabolic subgroups $P$ with
$P\supset A$, a property which plays an important role
in \cite{KSS}.

An interesting observation (which surprised us) is that 
in some cases there are also maximal split abelian 
subgroups $A$ for which 
$PH$ is open for all minimal parabolic subgroups $P$ with
$P\subset A$, but for which the polar decomposition fails
(see Remark \ref{surprise}). 

The fact that the triple space of $\SL(2,\C)$ 
admits open $P$-orbits follows from 
\cite{Kraemer} p.~152. A homogeneous space 
of algebraic groups over $\C$ with an
open Borel orbit is said to be spherical,
cf~\cite{Brion}, and the spaces we consider may be seen
as prototypes of spherical spaces over $\R$. 

In a final section we introduce an infinitesimal version
of the polar decomposition, and show that in the current 
setting it is valid if and only if the global
polar decomposition $G=KAH$ is valid.

The harmonic analysis on $\SL(2,\R)$ is an essential example
for understanding the harmonic analysis on general reductive
groups. We expect the triple spaces considered here
to serve similarly for the
harmonic analysis on non-symmetric homogeneous spaces, 
which is yet to be developed.

\section{Notation}

Let $\gf_0=\kf_0\oplus\sf_0$ be a Cartan decomposition of 
the Lie algebra $\gf_0$ of $G_0$,
and put $$\kf=\kf_0\times\kf_0\times\kf_0, \quad
\sf=\sf_0\times\sf_0\times\sf_0,$$ then $\gf=\kf\oplus\sf$ is 
also a Cartan decomposition.
The maximal abelian subspaces of $\sf$ have the form
\begin{equation}\label{three a's}
\af=\af_1\times\af_2\times\af_3
\end{equation}
with three maximal abelian subspaces $\af_1,\af_2,\af_3$ in $\sf_0$.

If for each $j$ we let $A_j=\exp\af_j$ and choose a positive
system for the roots of $\af_j$, then with
$G_0=K_0A_jN_j$ for $j=1,2,3$ we obtain the Iwasawa decomposition
$G=KAN$
where 
$$K=K_0\times K_0\times K_0,\quad A=A_1\times A_2\times A_3,\quad 
N=N_1\times N_2\times N_3.$$
Likewise we obtain the minimal parabolic subgroup 
$$P=P_1\times P_2\times P_3 =MAN$$
where $M=M_1\times M_2\times M_3$ and each $P_j=M_j A_jN_j$ is
a minimal parabolic subgroup of $G_0$.

\section{Polar decomposition}

Let $G/H$ be a homogeneous space of a reductive group
$G$, and let $\gf=\kf\oplus \sf$ be a Cartan decomposition
of the Lie algebra of $G$.
A decomposition of $G$ of the form
\begin{equation}\label{polar dec}
G=KAH,
\end{equation}
with $A=\exp\af$, for an
abelian subspace $\af\subset\sf$, 
is said to be a {\it polar} decomposition.
If such a decomposition exists, then the homogeneous
space $G/H$ is said to be of {\it polar type}
(see \cite{KSS}). 

The fact that symmetric spaces are of polar type
implies in particular that every double space 
$G/H=(G_0\times G_0)/\diag(G_0)$ with $G_0$ a real reductive group 
admits a polar decomposition. Here we can take 
$$\af=\af_0\times\af_0$$
for a maximal abelian subspace $\af_0\subset\sf_0$
(in fact, it would suffice to take already
the antidiagonal of $\af_0\times\af_0$).
Then $A$ has the form $A_1\times A_2$ with $A_1=A_2$.
In contrast, triple spaces do not admit $G=KAH$
for $A=A_1\times A_2\times A_3$ if $A_1=A_2=A_3$:

\begin{lemma}\label{not KAH}
Let $G/H$ be the triple space of a non-compact
semisimple Lie group $G_0$. Let $\af_0\subset\sf_0$ 
be maximal abelian and let $A=A_0\times A_0\times A_0$.
Then $KAH$ is a proper subset of $G$.
\end{lemma}

\begin{proof}
Let $a_0\in A_0$ be a regular element. 
We claim that a triple $(g_1,g_2,g_3)=(g_1,a_0,e)$ belongs to 
$KAH$ only if $g_1\in K_0A_0$.
Assume $g_i=k_ia_ig$ for $i=1,2,3$ with
$k_i\in K_0$, $A_i\in A_0$ and $g\in G_0$. From 
$$a_0=g_2g_3^{-1}=k_2a_2a_3^{-1}k_3^{-1}$$
we deduce that $k_2=k_3$, and from the regularity of $a_0$ 
we then deduce that $k_3$ belongs to the normalizer
$N_{K_0}(\af_0)$ (see \cite{Knapp}, Thm. 7.39). Then
$$g_1=g_1g_3^{-1}=k_1a_1a_3^{-1}k_3^{-1}\in K_0A_0.$$
The lemma follows immediately.
\end{proof}

It was observed in \cite{KSS} that the triple spaces
for the groups considered in (\ref{G0}) are of polar type.
In the following theorem we determine, for these groups, 
all the maximal abelian subspaces $\af$ of $\gf$
for which (\ref{polar dec}) holds. 

\begin{theorem}\label{thm polar}
Let $G_0$ be one of groups {\rm (\ref{G0})} and 
$\af\subset\sf$ as in {\rm (\ref{three a's})}. Then $G=KAH$ 
if and only if $\af_1+\af_2+\af_3$ has
dimension two in $\gf_0$.

In particular, $G/H$ is of polar type.
\end{theorem}

We shall approach
$G=KAH$ by a geometric argument. Let $Z_0=G_0/K_0$ be the
Riemannian symmetric space associated with $G_0$, and
let $z_0=eK_0\in Z_0$ denote its origin. Recall that 
(up to covering) $G_0$
is the identity component of the group of isometries of $Z_0$. 
Then it is easily seen that $G=KAH$ is equivalent to the following:

\begin{property}\label{isometry}
For every triple $(z_1,z_2,z_3)$ of points $z_j\in Z_0$
there exist a triple $(y_1,y_2,y_3)$ of points $y_j\in Z_0$
with $y_j\in A_jz_0$ for each $j$, and an isometry $g\in G_0$
such that $gz_j=y_j$ for $j=1,2,3$.
\end{property}

In order to illustrate the idea of proof, let us first state
and prove a Euclidean analogue.

\begin{proposition}\label{Euclidean} 
Let $\ell_1, \ell_2, \ell_3 \subset\R^n$ 
be lines 
through the origin $O$. The following statements are
equivalent
\begin{enumerate}
\item $\dim(\ell_1+\ell_2+\ell_3)=2$
\item For every triple of points $z_1, z_2, z_3\in\R^2$
there exists a rigid motion $g$ of $\R^n$
with $g(z_j)\in\ell_j$ for each $j=1,2,3$.
\end{enumerate}
\end{proposition}

\begin{proof} (1)$\Rightarrow$(2). Since the group of rigid
motions is transitive on the 2-planes in $\R^n$, we may assume
that $z_1$, $z_2$ and $z_3$ belong to the subspace spanned
by the lines. This reduces the proof to the case $n=2$.

We shall assume the $z_j$ are distinct as 
otherwise the result is easily seen. 
Furthermore, as at most two of the lines are identical,
let us assume that $\ell_1\neq \ell_2$.
Let $d$ denote the distance between $z_1$ and $z_2$,
and consider the set $\mathfrak X$
of pairs $(X_1,X_2)$ of points $X_1\in\ell_1$
and $X_2\in\ell_2$ with distance $d$ from each other.
Let $D_1$ be a point on $\ell_1$
with distance $d$ to the origin, then $(D_1,O)$
and $(-D_1,O)$ belong to $\mathfrak X$, and it follows from the 
geometry that we can connect these points by a continuous
curve $s\mapsto(X_1(s),X_2(s))$ in $\mathfrak X$,
say with $s\in [-1,1]$. For example, we can arrange that first
$X_1(s)$ moves from $-D_1$ to $O$ along $\ell_1$, 
while at the same time $X_2(s)$ moves along $\ell_2$ 
at distance $d$ from $X_1(s)$. Then $X_2(s)$ moves
from $O$ to a point $D_2\in\ell_2$ at distance $d$ from $O$. 
After that, $X_1(s)$ moves from $O$ to $D_1$,
while $X_2(s)$ moves back from $D_2$ to~$O$.

When $s$ passes through the interval $[-1,1]$, 
the line segment from $X_1(s)$ to $X_2(s)$ slides 
with its endpoints on the two lines. We define 
$X_3(s)$ such that the three points form a triangle
congruent to the one formed by $z_1$, $z_2$ and $z_3$.
In other words, for each $s\in[-1,1]$ there exists a 
unique rigid motion 
$g_s$ of $\R^n$ for which $g_s(z_1)=X_1(s)$ and
$g_s(z_2)=X_2(s)$. We let $X_3(s)=g_s(z_3)$. 
See the following figure. 
\newpage

\centerline{\includegraphics{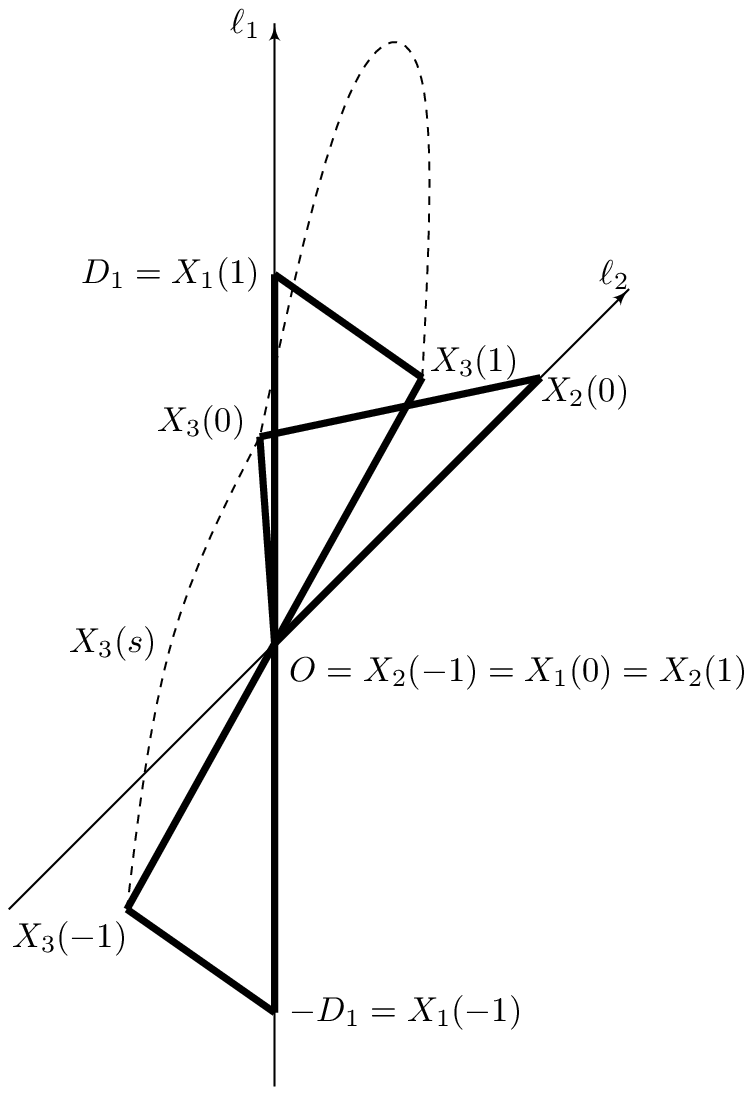}}

As $X_1(s)$ and $X_2(s)$ depend continuously
on $s$, then so does $g_s$ (in the standard topology of
the group of rigid motions) and hence also $X_3(s)$.
Since $X_1(\pm1)$ are opposite points while
$X_2(\pm1)=O$, the points $X_3(\pm 1)$ must
be opposite as well. Since $s\mapsto X_3(s)$ 
is a continuous curve that connects two opposite
points, it intersects with every line
through $O$. Let $s\in [-1,1]$ be a parameter value 
for which $X_3(s)\in \ell_3$. Now $g_s$ is the
desired rigid motion.

(2)$\Rightarrow$(1). Note that a rigid motion maps
affine lines to affine lines.
If $\dim(\ell_1+\ell_2+\ell_3)=1$
then $\ell_1=\ell_2=\ell_3$, and it is clear that
only triples of points which are positioned in a common 
affine line can be brought into it by a rigid motion. 
Hence $\dim(\ell_1+\ell_2+\ell_3)=1$ is excluded.

Let $z_1,z_2,z_3$ be an arbitrary
triple of distinct points located on a common affine line
$\ell$,
and let  $g$ be a rigid motion which brings these points 
into the $\ell_j$. Then $O$ can be 
one of the points $g(z_j)$, or not.
In the first case, say if $g(z_1)=O$, it follows that 
$\ell_2$ and $\ell_3$ are both equal to $g(\ell)$,
since each of these lines have two points in common with 
$g(\ell)$. Hence $\dim(\ell_1+\ell_2+\ell_3)\le 2$. 
In the second case, the line $g(\ell)$ together with $O$
spans a 2-dimensional subspace of $\R^n$,
which contains all the lines $\ell_j$.
Hence again $\dim(\ell_1+\ell_2+\ell_3)\le 2$.
\end{proof}

We proceed with the proof of Theorem \ref{thm polar}.

\begin{proof} 
Note that $\SL(2,\R)$ and $\SL(2,\C)$ are locally isomorphic to
$\SO_e(2,1)$ and $\SO_e(3,1)$, respectively.
The centers of $\SL(2,\R)$ and $\SL(2,\C)$ belong to
$K$, and hence $G=KAH$ will hold for the triple
spaces of these groups if and only
if it holds for the triple spaces of 
their adjoint groups. Thus
it suffices to consider $G_0=\SO(n,1)$ with $n\ge 2$.

The elements in $\so(n,1)$ have the form
\begin{equation}\label{Xmatrix}
X=\begin{pmatrix} A&b\\b^t&0 \end{pmatrix}
\end{equation}
where $A\in\so(n)$ and $b\in\R^n$, and 
$\sf_0$ consists of the elements with $A=0$.

Assume first that $\af_1+\af_2+\af_3$ is 2-dimensional.
By transitivity of the action of $K_0=\SO(n)$ on
the 2-dimensional subspaces of $\R^n$ we may 
assume that $\af_1+\af_2+\af_3$ consists of the
matrices $X$ as above with $A=0$ and 
$b$ non-zero only in the last two coordinates. 
Hence $\af_1+\af_2+\af_3$ is contained in
the $\so(2,1)$-subalgebra in
the lower right corner of $\so(n,1)$.
It follows that 
$\exp(\af_1+\af_2+\af_3).z_0$
is a 2-dimensional totally geodesic submanifold
of $Z_0$.

Let $z_1,z_2,z_3\in Z_0$ be given. Every triple
of points in $Z_0$ belongs to a 2-dimensional
totally geodesic submanifold $Z_0'$ of $Z_0$. 
For example, in 
the model of $Z_0$ as a one-sheeted hyperboliod 
in $\R^{n+1}$, we can obtain $Z_0'$ as the 
intersection of $Z_0$ 
with a 3-dimensional subspace of $\R^{n+1}$
containing the three points. Since $G_0$ is
transitive on geodesic submanifolds, we may
assume that $z_1,z_2,z_3$ are contained in
the submanifold generated by $\af_1+\af_2+\af_3$.
We have thus essentially reduced to the case 
$n=2$, and shall assume $n=2$ from now on.

We proceed exactly as in the Euclidean case 
and produce a pair of points $X_1(s)$ and
$X_2(s)$ on the geodesic lines $\exp(\af_1).z_0$ and 
$\exp(\af_2).z_0$, respectively. The two points are chosen
so that they have the same non-Euclidean distance 
from each other as $z_1$ and $z_2$, 
and they depend continuously on $s\in[-1,1]$.
Moreover, $X_1(-1)$ and $X_1(1)$ are symmetric with respect
to $z_0$, while $X_2(-1)=X_2(1)=z_0$.
As $Z_0$ is two-point homogeneous,
there exists for each $s\in [-1,1]$
a unique isometry $g_s\in G_0$
such that $g_s(z_j)=X_j(s)$ for $j=1,2$.
As before, a value of $s$, where the continuous
curve $s\mapsto g_s(z_3)$ intersects $\exp(\af_3)$,
produces the desired isometry $g_s$ of Property
\ref{isometry}. Hence $G=KAH$.

We return to the case $n\ge 2$ and
assume conversely that $G=KAH$. It follows from
Lemma \ref{not KAH} that $\dim(\af_1+\af_2+\af_3)>1$.
We want to exclude $\dim(\af_1+\af_2+\af_3)=3$.
Again we follow the Euclidean proof and select an arbitrary
triple of distinct points $z_1,z_2,z_3$ on a single geodesic $\gamma$
in $Z_0$. Then there is $g\in G_0$ such that
$gz_j=y_j$ for some $y_j\in\exp(\af_j).z_0$,
for $j=1,2,3$. If one of the $y_j$'s,
say $y_1$, is $z_0$, then $\exp(\af_2).z_0=\exp(\af_3).z_0=g(\gamma)$
and hence $\af_2=\af_3$. Otherwise, the geodesic $g(\gamma)$ is contained,
together with $O$, in a 2-dimensional totally geodesic 
submanifold of $Z_0$. This submanifold necessarily contains
the geodesic $\exp(\af_j).z_0$ for each $j$. 
Hence $\dim(\af_1+\af_2+\af_3)\le 2$.
\end{proof}

\section{Uniqueness}

If $G/H$ is a homogeneous space of polar type, so that
every element $g\in G$ allows a decomposition
$g=kah$, it is of interest to know to which extend
the components in this decomposition are unique. 
An obvious non-uniqueness is caused by the normalizer
$N_{K\cap H}(\af)$ of $\af$ in $K\cap H$, which
acts on $A$ by conjugation.
In the case of a symmetric space, it is known 
(see \cite{Sbook}, Prop.~7.1.3)
that the $A$ component of every $g\in G$
is unique up to such conjugation.
For our current triple spaces the description of 
which elements in $A$ generate the same
$K\times H$ orbit appears to be more complicated,
unless $\af_1= \af_2\perp \af_3$.

\begin{theorem}\label{unique} 
Let $G/H$ be the triple space with
$G_0$ as in {\rm (\ref{G0})}, and let 
$\af$ be as in {\rm (\ref{three a's})} with
$\af_1=\af_2\perp\af_3$. 
Let $a=(a_1,a_2,a_3)\in A$ with $a_1\neq a_2$ and let
$a'=(a_1',a_2',a_3')\in A$. 
Then $KaH=Ka'H$ if and only if
$a$ and $a'$ are conjugate by $N_{K\cap H}(\af)$.
\end{theorem}

We first determine
explicitly which pairs of elements $a,a'\in A$ are 
$N_{K\cap H}(\af)$-conjugate when $\af_1=\af_2\perp\af_3$.

\begin{lemma}\label{normalizer} 
Let $\af$ be as above. 
Then $a,a'\in A$ are conjugate by $N_ {K\cap H}(\af)$ if and only if
\begin{enumerate}
\item $(a_1',a_2')=(a_1,a_2)^{\pm 1}$ and $a_3'=a_3^{\pm 1}$
if $n>2$
\item $(a_1',a_2',a_3')=(a_1,a_2,a_3)^{\pm 1}$
if $n=2$.
\end{enumerate}
\end{lemma}

\begin{proof}
The normalizer  $N_{K\cap H}(\af)$ consists of all
the diagonal elements $k=(k_0,k_0,k_0)\in G$ for which 
$$k_0\in N_{K_0}(\af_1)\cap N_{K_0}(\af_2)\cap N_{K_0}(\af_3).$$
As elements $a_j,a_j'\in A_j$ are
$N_{K_0}(\af_j)$-conjugate if and only if $a_j'=a_j^{\pm1}$,
only the pairs mentioned under 
(1) can be conjugate when $\af_1=\af_2$.

Let $\delta,\epsilon=\pm1$. For the groups in (\ref{G0})
the adjoint representation is surjective $K_0\to\SO(\sf_0)$.
If $n>2$ then there exists a transformation in
$\SO(\sf_0)$ which acts by 
$\delta$ on $\af_1=\af_2$ and by $\epsilon$ on $\af_3$.
Its preimages in $K_0$ conjugate
$(a_1,a_2,a_3)$ to $(a_1^{\delta},a_2^{\delta},a_3^{\epsilon})$.
When $n=2$ such a transformation exists if and only if 
$\delta=\epsilon$. The lemma follows.
\end{proof}

The following lemmas are used in the proof
of Theorem \ref{unique}. Here
$G_0$ can be any real reductive group
with Cartan decomposition $\gf_0=\kf_0+\sf_0$.

\begin{lemma}\label{XUX} 
Let $X,U\in\sf_0$. Then $\exp X\exp U\exp X\in\exp\sf_0$.
\end{lemma}

\begin{proof} Let $\theta$ denote the Cartan involution
and note that the product $\exp (tX)\exp (tU)\exp (tX)$ belongs to 
$S=\{ g\in G_0\mid \theta(g)=g^{-1}\}$ for all $t\in[0,1]$. 
It is easily seen that $k\exp Y\in S$ implies $k^2=e$
for $k\in K_0$ and $Y\in\sf_0$, and since $e$ is isolated
in the set of elements of order $2$ it follows that
$\exp\sf_0$ is the identity component of $S$. 
Hence $\exp X\exp U\exp X\in\exp\sf_0$.
\end{proof}

\begin{lemma}\label{two items} 
Let $\af_0\subset\sf_0$ be a one-dimensional 
subspace and let $A_0= \exp\af_0$.
\begin{enumerate}
\item If $g\in \exp\sf_0$ and $ga_0\in a_0'K_0$ for 
some $a_0,a_0'\in A_0$, then $g=a_0'a_0^{-1}$.
\item
If $g\in G_0$ and $ga_1,ga_2\in A_0K_0$
for some $a_1,a_2\in A_0$ with $a_1\neq a_2$ then
$g\in N_{K_0}(\af_0)A_0$.
\end{enumerate}
\end{lemma}

\begin{proof} (1) It follows from
$ga_0\in a_0'K_0$ that $a_0ga_0\in a_0a_0'K_0$. 
Since $a_0ga_0\in \exp\sf_0$ by Lemma \ref{XUX}, 
it follows from uniqueness of the Cartan decomposition that 
$a_0ga_0=a_0a_0'$ and thus $g=a_0'a_0^{-1}$.

(2) Put $z_0=eK_0$, then $A_0. z_0$ is a geodesic in $G_0/K_0$.
Since $g$ maps two distinct points on  
$A_0. z_0$ into $A_0. z_0$, it maps the entire
geodesic onto itself, and hence so does $g^{-1}$. In particular 
$g^{-1}.z_0\in A_0K_0$, that is, $g=k_0a_0$
for some $k_0\in K_0$, $a_0\in A_0$. It follows for all $a\in A_0$ that
$$k_0ak_0^{-1}=ga_0^{-1} ak_0^{-1}\in gA_0K_0=A_0K_0.$$ 
As $k_0ak_0^{-1}\in\exp\sf_0$, 
uniqueness of the Cartan decomposition
implies  $k_0ak_0^{-1}\in A_0$,
i.e. $k_0\in N_{K_0}(\af_0)$.
\end{proof}

\begin{lemma}\label{third item}
Let $\af_1,\af_3\subset\sf_0$ be one-dimensional subspaces
with  $\af_1\perp\af_3$ 
and let $A_1= \exp\af_1$, $A_3= \exp\af_3$.
If $g\in N_{K_0}(\af_1)A_1$ and 
$ga_3\in a_3'K_0$ for some $a_3,a_3'\in A_3$,
not both equal to $e$, then $g\in N_{K_0}(\af_1)
\cap N_{K_0}(\af_3)$.
\end{lemma}

\begin{proof}
We may assume $a_3'\neq e$, as otherwise we interchange 
it with $a_3$ and replace $g$ by $g^{-1}$. We
consider the geodesic triangle in $G_0/K_0$ formed by 
the geodesics
$$
L_1:=A_1.z_0 ,\quad
L_2:=A_3.z_0, \quad
L_3:=gA_3.z_0.
$$
The vertices are
$$D_3:=z_0,\quad D_2:=g.z_0,\quad D_1:=ga_3.z_0=a_3'.z_0.$$
As $L_1$ and $L_2$ intersect orthogonally, angle $D_3$ is right.
The isometry $g$ maps $L_1$ to itself and
$L_2$ to $L_3$. Hence $L_1$ and $L_3$ also intersect
orthogonally and angle $D_2$ is right. 
As the sectional curvature
of $G_0/K_0$ is $\le 0$, it is impossible 
for a proper triangle to have two right angles. 
As $L_1\neq L_2$ and $D_3\neq D_1$ we conclude
$D_3=D_2$ and $L_3=L_2$. It follows that $g\in K_0$ and by
Lemma \ref{two items} (2) that $g\in N_{K_0}(\af_3)$.
\end{proof}

\begin{proof}[Proof of Theorem \ref{unique}]
Assume $KaH=Ka'H$. Then $Kah=Ka'$ for some $h=(g,g,g)\in H$.
Applying Lemma \ref{two items} (2) to the first two coordinates
of $Kah=Ka'$ we conclude that $g\in N_{K_0}(\af_1)A_1$.

If $a_3'$ and $a_3$ are not both $e$, we can 
apply Lemma \ref{third item} to the last coordinate
and conclude $g\in N_{K_0}(\af_1)\cap N_{K_0}(\af_3)$.
Hence $h\in N_{K\cap H}(\af)$, and we conclude
that $a'=h^{-1}ah$.

If $a_3'=a_3=e$ it follows from the third
coordinate that $g\in K_0$. Hence $g\in N_{K_0}(\af_1)$
and $a'=a$ or $a'=a^{-1}$.
\end{proof}

\begin{rmk}
When $\dim\sf_0=2$ the assumption in 
Theorem \ref{unique} and Lemmas \ref{normalizer}, \ref{third item},
that $\af_1=\af_2\perp\af_3$, can be relaxed to 
$\af_1=\af_2\neq\af_3$ with unchanged conclusions.
This follows from the fact that in a two dimensional space
the only proper orthogonal transformations which
normalize a one-dimensional subspace are $\pm I$.
Hence $N_{K_0}(\af_1)=N_{K_0}(\af_3)$ in this case.
\end{rmk}

\section{A formula for the invariant measure}

In a situation where there is uniqueness (up to some well-described
isomorphism), it is of interest to explicitly determine the invariant 
measure with respect to the $KAH$-coordinates. 

For any triple space $G/H$ of 
a unimodular Lie group $G_0$ we note that the map
\begin{equation}\label{GtimesG}
G_0\times G_0\to G/H, \ \ (g_1, g_2)\mapsto (g_1, g_2, e) H
\end{equation}
is a $G_0\times G_0$-equivariant diffeomorphism. Accordingly the
invariant measure on $G/H$ identifies with the Haar measure on
$G_0\times G_0$.

For $G_0=\SO_e(n,1)$ we define  
$X\in\so(n,1)$ by (\ref{Xmatrix}) with $A=0$ and
$b=e_n$, and $Y\in\so(n,1)$ 
similarly with $A=0$ and $b=e_1$. 
Let $\af_1=\af_2=\R X$ and $\af_3=\R Y$, 
then $\af_3\perp \af_1$. Let
$$a_t=\exp(tX)\in A_1=A_2,\quad b_s=\exp(sY)\in A_3.$$

\begin{lemma} 
Let $G/H$ be the triple space of $G_0=\SO_e(n,1)$ and let 
$\af_1=\af_2$ and $\af_3$ be as above.
Consider the polar coordinates
\begin{equation}\label{coordinates}
K\times\R^3\ni
(k,t_1,t_2,s)\mapsto (k_1a_{t_1},k_2a_{t_2},k_3b_s)H
\end{equation}
on $G/H$. The invariant measure $dz$ of $G/H$ can be 
normalized so that in these coordinates
\begin{equation}\label{measure}
dz =  J(t_1, t_2, s)\, dk\, dt_1\, dt_2\, ds
\end{equation}
where $dk$ is Haar measure, $dt_1, dt_2, ds$ 
Lebesgue measure, and where 
$$J(t_1,t_2,s)=|\sinh^{n-1}(t_1-t_2)\sinh^{n-2}(s)\cosh(s)|.$$
\end{lemma}

\begin{proof}
On $G_0 \times G_0$ we use the formula (see \cite{Sbook}, Thm. 8.1.1)
for integration in $KAH$
coordinates for the symmetric space $G_0
\times G_0 /\diag(G_0)= G_0$. The map
$$(K_0\times K_0)\times A_0 \times G_0 \to G_0\times G_0$$
defined by
$$(k,a_t, g)\mapsto (k_1a_{t/2}g,k_2a_{-t/2}g) $$
is a parametrization (up to the sign of $t$), 
and the Haar measure on $G_0\times G_0$
writes as
\begin{equation}\label{inv1}
|\sinh^{n-1}(t)| \, dk_1 \, dk_2 \,  dt\,  dg \, .
\end{equation}

Further we decompose the diagonal copy of $G_0$ by means of the
$HAK$ coordinates for the symmetric space $G_0/(\SO(n-1)\times A_1)$,
where $\SO(n-1)$ is located in the upper left corner of $G_0$.
Note that the subgroup $A_3$ serves as the `$A$' 
in this decomposition. In the coordinates
$$ K_0\times A_3\times \SO(n-1)\times A_1 \to G_0, \ \ (k_3, b_s, m, a_u)\mapsto
a_u m b_s k_3$$
we obtain (again using \cite{Sbook}, Thm. 8.1.1),
\begin{equation}\label{inv2}
dg= |\sinh^{n-2}(s)\cosh(s)| \, dk_3 \, db_s \, dm\, du\, .
\end{equation}
Combining (\ref{inv1}) and (\ref{inv2}), 
we have the coordinates $$(k_1 a_{u+t/2} m b_s  k_3, k_2
a_{u-t/2} m b_s k_3)$$ on $G_0 \times G_0,$ with Jacobian
$|\sinh^{n-1}(t) \sinh^{n-2}(s)\cosh(s)|$. 
As the subgroup $\SO(n-1)$ centralizes $A_1$, the 
integration over $m$ is swallowed by the integrations over $k_1$ and $k_2$.
Changing coordinates $u$, $t$ to $t_1=u+t/2$ and $t_2=u-t/2$ 
we find $t=t_1-t_2$.

Finally we apply (\ref{GtimesG})
so that the above coordinates correspond to
$$(k_1,k_2,k_3)(a_{t_1},a_{t_2},b_{-s})\diag(G_0)$$
This proves (\ref{measure}).
\end{proof}

\section{Spherical decomposition}

A decomposition of $\gf$ of the form
\begin{equation}\label{sph dec}
\gf=\pf+\hf
\end{equation}
with $\pf$ a minimal parabolic subalgebra
is said to be a {\it spherical} decomposition.
If such a decomposition exists, then the homogeneous
space $G/H$ is said to be of {\it spherical type}
(see \cite{KSS}). 

Note that with $\gf_0=\so(n,1)$ we have
(see (\ref{dim}) and (\ref{1par}))
$$\dim\pf+\dim\hf-\dim\gf=\tfrac12(n^2-5n+6)\ge 0.$$
In particular spherical decompositions 
will be direct sums if $n=2,3$. 

It was observed in \cite{KSS} that the triple spaces
for the groups considered in (\ref{G0}) are of spherical type.
In the following we determine for each $n$ all the minimal
parabolic subalgebras $\pf$ for which (\ref{sph dec}) holds.

\begin{prop}\label{thm sph}
Let $G_0$ be one of the groups {\rm (\ref{G0})} and
let $\pf=\pf_1\times \pf_2\times \pf_3$ 
a minimal parabolic subalgebra. Then
$\gf=\pf+\hf$ holds if and only if 
$\pf_1$, $\pf_2$ and $\pf_3$ are distinct.

In particular, the triple space $G/H$ is of spherical type
for all groups $G_0$ in {\rm (\ref{G0})}.
\end{prop}

We prepare by the following lemma.

\begin{lemma} \label{3subspaces}
Let $U_1,U_2,U_3\subset V$ be subspaces of a vector space $V$.
Put 
$$U:=U_1\times U_2\times U_3\subset X:=V\times V\times V,$$
and
$Y:=\diag(V)\subset X.$
Then $X=U+Y$ if and only if 
\begin{equation}\label{three statements}
V=U_1+(U_2\cap U_3)=U_2+(U_3\cap U_1)=U_3+(U_1\cap U_2).
\end{equation}
\end{lemma}

\begin{proof} Assume first that $X=U+Y$ and let $v\in V$ be given.
Writing $$(v,0,0)=(u_1,u_2,u_3)+\diag(w)$$ we see that
$w=-u_2=-u_3\in U_2\cap U_3$, and hence $v=u_1+w\in U_1+(U_2\cap U_3)$.
The other two statements in (\ref{three statements}) follow similarly.

Conversely, we assume 
(\ref{three statements}) and let $x=(x_1,x_2,x_3)\in X$ be given.
We decompose $x_1$, $x_2$ and $x_3$ according to the three 
decompositions in
(\ref{three statements}), that is,
\begin{align*}
x_1&=u_1+t_1, \qquad u_1\in U_1,\,\, t_1\in U_2\cap U_3\\
x_2&=u_2+t_2, \qquad u_2\in U_2,\,\, t_2\in U_3\cap U_1\\
x_3&=u_3+t_3, \qquad u_3\in U_3,\,\, t_3\in U_1\cap U_2.
\end{align*} 
Then 
$$x=(u_1-t_2-t_3,u_2-t_1-t_3,u_3-t_1-t_2)+\diag(t_1+t_2+t_3)$$
is a decomposition of the desired form $U+Y$.
\end{proof}

\begin{rmk} In fact, it is easily seen that
any two of the decompositions of $V$ in
(\ref{three statements}) together imply the third.
\end{rmk}

\begin{proof}[Proof of Proposition \ref{thm sph}] It suffices to consider 
$G_0=\SO_0(n,1)$ because of the local isomorphisms.

If for example $\pf_1=\pf_2$ then $\pf_1+(\pf_2\cap\pf_3)=\pf_1$.
Hence $\pf_1+(\pf_2\cap\pf_3)$ is proper in $\gf_0$ and it follows from
Lemma \ref{3subspaces} that $\gf=\pf+\hf$
fails to hold. This implies one direction of the first statement.

For the other direction
it follows from Lemma \ref{3subspaces} 
that it suffices to prove
$$\gf_0=\pf_1+(\pf_2\cap\pf_3)$$
for all triples of
distinct parabolics in $\so(n,1)$.
We shall do this by proving
\begin{equation}\label{dims}
\dim\gf_0=\dim\pf_1+\dim(\pf_2\cap\pf_3)-\dim(\pf_1\cap\pf_2\cap\pf_3).
\end{equation}
We find 
\begin{equation}\label{dim}
\dim\gf_0=\dim\so(n,1)=\tfrac12(n^2+n),
\end{equation} 
and claim that
\begin{align}
\label{1par}&\dim\pf_1=\tfrac12(n^2-n+2)\\
\label{2par}&\dim(\pf_1\cap\pf_2)=\tfrac12(n^2-3n+4)\\
\label{3par}&\dim(\pf_1\cap\pf_2\cap\pf_3)=\tfrac12(n^2-5n+6).
\end{align}
The equations (\ref{dim})-(\ref{3par})
imply (\ref{dims}). 

The parabolic subalgebras $\pf$ of $\so(n,1)$ are the normalizers of 
the isotropic lines in $\R^{n+1}$, that is, the one-dimensional subspaces 
of the form
$L_q=\R(q,1)$ where $q\in\R^n$ with $\|q\|= 1$. 

Recall that all elements in $\so(n,1)$ have the form
(\ref{Xmatrix})
with $A\in\so(n)$ and $b\in\R^n$. 
It follows that $X\in\pf$ if and only if
\begin{equation}\label{condition}
Aq+b=(b\cdot q)q.
\end{equation}

Let us prove (\ref{1par}). Let $q_1$ be the unit vector
such that $\pf_1$ is the stabilizer of $L_{q_1}$,
and extend $q_1$ to a basis $q_1,\dots,q_n$ for $\R^n$.
For $b\in\R^n$ we let $x_1=(b\cdot q_1)q_1-b$ 
and we observe that $x_1\cdot q_1=0$. According to
(\ref{condition}) the matrix $X$ of (\ref{Xmatrix}) 
belongs to $\pf_1$ if and only if $Aq_1=x_1$. In order to satisfy
that we can define an $n\times n$ matrix $A$ by
\begin{equation}\label{1cases}
Aq_i\cdot q_j:= 
\begin{cases}
x_1\cdot q_j &\text{for } i=1\\
-x_1\cdot q_i &\text{for } j=1\\
a_{ij} &\text{for } i,j > 1
\end{cases}
\end{equation}
with arbitrary antisymmetric entries in the last line.
Then $A\in\so(n)$ and $Aq_1=x_1$.
The degree of freedom for each $b$ is
$$\dim\so(n-1)=\tfrac12(n-1)(n-2),$$
and hence $\dim\pf_1=n+\frac12(n-1)(n-2)=\frac12(n^2-n+2)$ as asserted.

Next we prove (\ref{2par}).
Let $q_1,q_2$
be the unit vectors such that
$\pf_i$ is the stabilizer of $L_{q_i}$.
By assumption $q_1\neq q_2$.
For the element $X$ of (\ref{Xmatrix}) to be in
$\pf_1\cap\pf_2$ we need that (\ref{condition}) 
is satisfied in both cases, that is,
\begin{equation}\label{2eq}
Aq_i=x_i,\quad (i=1,2).
\end{equation}
where $x_i=(b\cdot q_i)q_i-b$.
Now
$$x_2\cdot q_1+x_1\cdot q_2=(q_1\cdot q_2-1)(b\cdot(q_1+q_2)).$$ 
Note that $q_1\cdot q_2<1$ since $q_1\neq q_2$. As
$A\in\so(n)$ we conclude that
$$b\cdot(q_1+q_2)= 0$$
since otherwise (\ref{2eq}) would lead to contradiction.

Conversely, let $b\in\R^n$ be such that $b\cdot(q_1+q_2)=0$ 
and define $x_1,x_2$ by $x_i=(b\cdot q_i)q_i-b$. 
Then $x_i\cdot q_j=-x_j\cdot q_i$ for all pairs
$i,j\le 1,2$. 
We extend $q_1,q_2$ to a basis and
define an $n\times n$ matrix $A$ by
\begin{equation}
Aq_i\cdot q_j= 
\begin{cases}
x_i\cdot q_j &\text{for } i=1,2\\
-x_j\cdot q_i &\text{for } j=1,2\\
a_{ij} &\text{for } i,j>2
\end{cases}
\end{equation}
with arbitrary antisymmetric entries in the last line.
Then $A\in\so(n)$ and  (\ref{2eq}) holds. 
The degree of freedom for each $b$ is 
$$\dim\so(n-2)=\tfrac12(n-2)(n-3)$$
and hence 
$\dim(\pf_1\cap\pf_2)=n-1+\frac12(n-2)(n-3)=\frac12(n^2-3n+4)$ as asserted.

Finally, to prove (\ref{3par}) assume that $X$ in (\ref{Xmatrix}) belongs to
$\pf_1\cap\pf_2\cap\pf_3$. As above, it follows that 
$$b\cdot(q_1+q_2)=b\cdot(q_1+q_3)=b\cdot(q_2+q_3)=0$$
which implies that $b\cdot q_i=0$ for $i=1,2,3$.
If this is satisfied by $b$, the condition (\ref{condition})
simplifies to 
\begin{equation}\label{3eq}
Aq_i=-b,\qquad i=1,2,3.
\end{equation}
We first assume that $q_1,q_2,q_3$ are linearly independent
and extend to a basis as before. Given $b\in\R^n$ such that $b\cdot q_i=0$ for
$i=1,2,3$ we define $A$ by
\begin{equation}
Aq_i\cdot q_j= 
\begin{cases}
-b\cdot q_j &\text{for } i=1,2,3\\
b\cdot q_i &\text{for } j=1,2,3\\
a_{ij} &\text{for } i,j>3
\end{cases}
\end{equation}
with arbitrary antisymmetric entries in the last line.
The degree of freedom for each $b$ is 
$$\dim\so(n-3)=\tfrac12(n-3)(n-4)$$
and hence 
$\dim(\pf_1\cap\pf_2\cap\pf_3)=n-3+\frac12(n-3)(n-4)=\frac12(n^2-5n+6)$ as asserted.

Next we assume linear dependence of $q_1,q_2,q_3$.
This implies a further obstruction to $b$. In fact, let 
$\lambda_1q_1+\lambda_2q_2+\lambda_3q_3=0$ be a non-trivial relation, then
it follows from (\ref{3eq}) that $(\lambda_1+\lambda_2+\lambda_3)b=0$.
Since $q_1,q_2,q_3$ are assumed to be distinct unit vectors
the sum of the $\lambda$'s cannot be zero, and we conclude that $b=0$.
Thus in this case the only freedom comes from the choice of $A$. 
That can be chosen
arbitrarily from the annihilator in $\so(n)$
of the space spanned by the three $q$'s.
We obtain  
$\dim(\pf_1\cap\pf_2\cap\pf_3)=\dim\so(n-2)=\frac12(n^2-5n+6)$
as before. This concludes the proof of (\ref{3par}).

In particular, if $\af_1$, $\af_2$ and $\af_3$ are all different,
then $\gf=\pf+\hf$ for every parabolic subalgebra $\pf$ above
$\af=\af_1\times\af_2\times\af_3$. Hence $G/H$ is of spherical type.
\end{proof}

\begin{cor}\label{strong spherical}
There exists a maximal abelian subspace $\af\subset\sf$ 
for which both 
\begin{enumerate}
\item[(i)]
the polar decomposition {\rm (\ref{polar dec})} is valid, and
\item[(ii)] the spherical decomposition {\rm (\ref{sph dec})} is 
valid for all minimal parabolic subalgebras containing $\af$.
\end{enumerate}
\end{cor}

\begin{proof} Let $\af_j\subset\sf_0$ for $j=1,2,3$ be mutually different
and with a two-dimensional sum. It follows from 
Theorem \ref{thm polar} and Proposition \ref{thm sph} that
$\af=\af_1\times\af_2\times\af_3$ satisfies (i) and (ii).
\end{proof}

\begin{rmk} \label{surprise}
The properties of a reductive homogeneous space
$G/H$ that it is of polar type, respectively of spherical type,
appear to be closely related. However, the relation is not
as strong as one might hope, because the conditions on $\af$
are different in Theorem \ref{thm polar} 
and Proposition \ref{thm sph}. In particular, there exist
maximal abelian subspaces $\af\subset \gf$ which fulfill
(ii) but not (i),
namely the 'most generic' ones,
for which $\dim(\af_1+\af_2+\af_3)=3$.
\end{rmk}

\section{Infinitesimal polar decomposition}

Here we consider an infinitesimal version of the
polar decomposition $G=KAH$. Let $G/H$ be a homogeneous space of a reductive group
$G$, and let $\gf=\kf+\sf$ be a Cartan decomposition.

\begin{definition} 
A decomposition of the form
\begin{equation}\label{IP} 
\sf = \Ad (K\cap H) \af+ \sf\cap \hf 
\end{equation}
with an abelian subspace $\af\subset\sf$
is called a  polar decomposition.

If there exists such a decomposition of $\sf$ then 
we say that $G/H$ is infinitesimally polar.
\end{definition}

Here 
$$\Ad (K\cap H) \af=\{ \Ad(k)X \mid k\in K\cap H, X\in \af\}.$$
Note that this need not be a vector subspace of $\sf$. 

If $G/H$ is a symmetric space, then we can choose the Cartan decomposition 
so that $\kf$ and $\sf$ are stable under the involution $\sigma$ that determines
$G/H$. If $\gf=\hf+\qf$ denotes the decomposition of $\gf$ in $+1$ and $-1$
eigenspaces for $\sigma$, then $\sf=\sf\cap\qf+\sf\cap\hf$. If furthermore
$\af_q$ is a maximal abelian subspace of $\sf\cap\qf$, then it is known
that $\sf\cap\qf=\Ad (K\cap H) \af_q$ and hence (\ref{IP}) follows.

The following lemma suggests that there is a close
connection between polar decomposability
and infinitesimally polar decomposability.

\begin{lemma} 
Let $G_0$ be one of groups {\rm (\ref{G0})} and let
$\af=\af_1\times \af_2\times \af_3$. Then the infinitesimal
polar decomposition {\rm (\ref{IP})} holds if and only if 
$\dim(\af_1+\af_2+\af_3)=2$.
\end{lemma}

\begin{proof} For the triple spaces, the polar decomposition
(\ref{IP}) interprets to the statement that for every
triple of points $Z_1,Z_2,Z_3\in\sf_0$ there exist
$k\in K_0$, $T\in \sf_0$ and $X_j\in\af_j$ ($j=1,2,3$)  
such that $Z_j=\Ad(k)X_j+T$. As the maps
$X\mapsto \Ad(k)X+T$ with $k\in K_0$ and $T\in\sf_0$ are exactly
the rigid motions of $\sf_0$, this lemma is precisely the content of 
Proposition~\ref{Euclidean}. 
\end{proof}

Combining the lemma with Theorem \ref{thm polar} we see that
for our triple spaces
the infinitesimal polar decomposition holds with a  given $\af$ 
if and only if the global polar decomposition $G=KAH$ holds
for the corresponding $A=\exp\af$.


\begin{thebibliography} {10}

\bibitem{Brion} M. Brion, {\it Classification des espaces
homog\`enes sph\'eriques}, Compositio Math. {\bf 63}
(1987), 189--208.

\bibitem{FJ}
M. Flensted-Jensen,
{\it Spherical functions on rank one symmetric spaces and generalizations.} 
Proc. Sympos. Pure Math., Vol. XXVI, pp. 339--342. 
Amer. Math. Soc., Providence, R.I., 1973. 

\bibitem{Knapp}
A. W. Knapp, Lie groups beyond an introduction. Birkh\"auser, 2002.

\bibitem{Kraemer} M. Kr\"amer, {\it Sph\"arische Untergruppen in kompakten
zusammenh\"angenden Gruppen}, Compositio Math. {\bf 38} (1979), 129--153.

\bibitem{KSS}
B.  Kr\"{o}tz, H. Schlichtkrull and  E. Sayag,
{\it Decay of matrix coefficients on reductive homogeneous spaces of spherical 
type}, submitted.


\bibitem{Sbook} H. Schlichtkrull, {\it Hyperfunctions
and harmonic analysis on symmetric spaces},
Birkh\"auser 1984.




\bibitem{Wolf}
Wolf, Joseph A. {\it Finiteness of orbit structure for real flag manifolds.} 
Geometriae Dedicata {\bf 3} (1974), 377--384.


\end{thebibliography}
\end{document}